\documentclass[arxiv,reqno,twoside,a4paper,12pt]{amsart}

\usepackage[margin=3cm]{geometry}

\usepackage[scaled]{helvet}
\usepackage{courier}
\usepackage[mathbf]{euler}
\usepackage{caption}

\usepackage{amssymb, amsfonts, amsbsy, amsmath, amsthm, mathrsfs}
\usepackage[english]{babel}
\usepackage{listings}

\usepackage[german=quotes]{csquotes}

\usepackage{graphicx}
\usepackage{tikz}
\usepackage{float}
\usepackage{enumerate}

\definecolor{qqwuqq}{rgb}{0,0,0}

\usepackage{color}

\usepackage[colorlinks,linkcolor=blue, citecolor=blue,urlcolor=blue]{hyperref}

\usepackage{marginnote}
\usepackage[color]{changebar}
\cbcolor{red}

\newcommand{\N}{{\mathbb N}}

\newcommand{\R}{{\mathbb R}}
\newcommand{\C}{{\mathbb C}}

\newcommand{\Id}{\operatorname{Id}}

\theoremstyle{definition}
\newtheorem{definition}{Definition}[section]

\theoremstyle{plain}
\newtheorem{remark}[definition]{Remark}
\newtheorem{explanation}[definition]{Possible explanation}
\newtheorem{proposition}[definition]{Proposition}
\newtheorem{lemma}[definition]{Lemma}
\newtheorem{theorem}[definition]{Theorem}

\DeclareMathOperator{\Det}{det'}

\DeclareMathOperator{\tr}{Tr}

\DeclareMathOperator{\oo}{\mathcal{O}}

\numberwithin{equation}{section}

\begin{document}

\date{\today}
\title[Extremals of determinants for Laplacians on discrete surfaces]
{Extremals of determinants for Laplacians \\ on discrete surfaces}

\author{Paul Hafemann}
\address{University of Oldenburg, Germany} 
\email{paul.hafemann@uni-oldenburg.de} 

\author{Boris Vertman}
\address{University of Oldenburg, Germany} 
\email{boris.vertman@uni-oldenburg.de} 

\begin{abstract}
In this note we pursue a discrete analogue of a celebrated theorem by Osgood, 
Phillips and Sarnak, which states that in a fixed conformal class of Riemannian metrics 
of fixed volume on a closed Riemann surface, the zeta-determinant 
of the corresponding Laplace-Beltrami operator is maximized on a metric of 
constant Gaussian curvature. We study the corresponding question for the discrete cotan Laplace operator
on triangulated surfaces. We show for some types of triangulations that the determinant (excluding zero eigenvalues) of the
discrete cotan-Laplacian  is stationary at discrete metrics of constant discrete Gaussian curvature. 
We present some numerical calculations, which suggest that these stationary points are in fact minima, 
strangely opposite to the smooth situation. We suggest an explanation of this discrepancy.
\end{abstract}

\maketitle

\tableofcontents

\section{Introduction and statement of the main results}

In their celebrated theorem, Osgood, Phillips and Sarnak \cite{ops} have employed
Ricci flow to prove that in a fixed conformal class and for fixed volume, the zeta regularized determinant of the 
Laplace-Beltrami operator on a closed Riemann surface is maximal when the metric 
is of constant Gaussian curvature. Their result can be phrased as follows.

\begin{theorem}[Osgood, Philips and Sarnak] \label{OPS}
The zeta-determinant of the Laplace-Beltrami operator on a compact Riemann surface 
of fixed volume attains its local maxima at metrics of constant Gaussian curvature. 
\end{theorem} 

There are discrete counterparts to all of the objects appearing in this statement: the Riemann surface 
becomes a triangulation endowed with a discrete metric for which discrete Laplace operators can be defined. 
There are well-studied concepts like discrete Gaussian curvature on a triangulation. 
Thus it makes sense to ask if there is a discrete counterpart to the Osgood-Phillips-Sarnak theorem.
Our main observation is that two particular classes of triangulations admit an analogy to the smooth case, namely
triangulations defined by complete and by symmetric graphs. 

\subsection{Extremals of determinants for complete graphs}

Our first main result studies determinants on complete graphs. 
The full statement is presented below in Theorem \ref{sphth}
and reads informally as follows.

\begin{theorem}\label{main1}
Let $M = (V, E, T)$ be a triangulation of a smooth closed Riemannian surface, consisting of vertices $V$, edges $E$ and
triangle faces $T$. Consider a discrete metric $l: E \to (o,\infty)$ and impose the following conditions:
\begin{enumerate}
\item the graph is complete, i.e. all pairs of vertices are connected by an edge.
\item $M$ has no boundary, i.e. each edge is bounded by exactly two triangles.
\item the discrete metric $l$ induces angles $\pi/3$ in each triangle
\footnote{In particular, the discrete Gaussian curvature of $(M,l)$ is constant at all vertices.}.
\end{enumerate}
Then the determinant (excluding zero eigenvalues) 
of the cotan-Laplace operator $\Delta = \Delta_\text{cot}$ on $(M,l)$ is stationary 
with respect to variational changes of $l$. 
\end{theorem}

Examples that satisfy the conditions of the theorem include: a) 
a triangulation of the sphere $\mathbb{S}^2$: the tetrahedron, inducing the complete graph $K_4$;
b) the $7$-point triangulation of the $2$-torus $\mathbb{T}^2 = \mathbb{S}^1 \times \mathbb{S}^1$: the Császár polyhedron, 
inducing the complete graph $K_7$. Both triangulations are illustrated in Figure \ref{pyradjazenz}. 

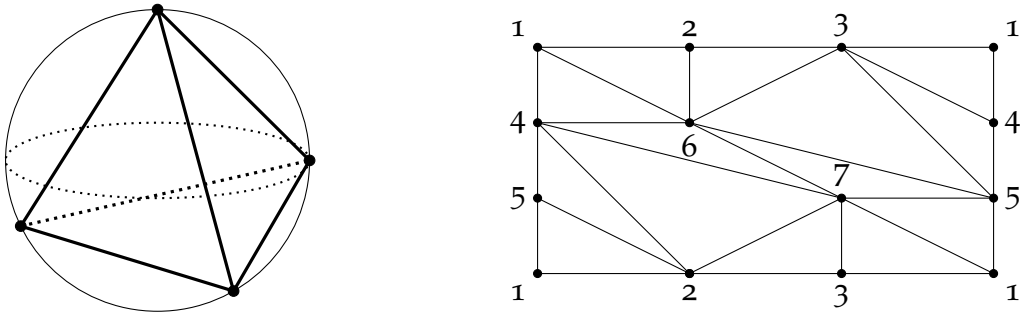
\begin{figure}[h]    
\centering
\begin{tikzpicture}[scale = 1]

\draw (0,0) circle (2);
\draw[dotted, thick] (0,0) ellipse (2 and 0.5);
\draw[very thick] (0,2) -- (-1.8,-0.871) -- (1,-1.732) -- (2,0) -- (0,2);
\draw[very thick] (0,2) -- (1,-1.732);
\draw[very thick, dotted] (-1.8,-0.871) -- (2,0);
\filldraw (0,2) circle (0.07);
\filldraw (-1.8,-0.871) circle (0.07);
\filldraw (1,-1.732) circle (0.07);
\filldraw (2,0) circle (0.07);

   \begin{scope}[shift={(5,-1.5)}]
       \draw (0,0) -- (6,0) -- (6,3) -- (0,3) -- (0,0);

\filldraw (0,0) circle (1.5pt);
\filldraw (2,0) circle (1.5pt);
\filldraw (4,0) circle (1.5pt);
\filldraw (6,0) circle (1.5pt);
\filldraw (0,1) circle (1.5pt);
\filldraw (0,2) circle (1.5pt);
\filldraw (0,3) circle (1.5pt);
\filldraw (2,0) circle (1.5pt);
\filldraw (2,2) circle (1.5pt);
\filldraw (2,3) circle (1.5pt);
\filldraw (4,1) circle (1.5pt);
\filldraw (4,3) circle (1.5pt);
\filldraw (6,1) circle (1.5pt);
\filldraw (6,2) circle (1.5pt);
\filldraw (6,3) circle (1.5pt);

\draw (0,3) -- (2,2) -- (0,2) -- (2,0) -- (0,1);
\draw (6,0) -- (4,1) -- (6,1) -- (4,3) -- (6,2);
\draw (0,2) -- (4,1) -- (2,0);
\draw (6,1) -- (2,2) -- (4,3);
\draw (2,3) -- (2,2) -- (4,1) -- (4,0);
\node at (0,0)[below left] {1};
\node at (6,0)[below right] {1};
\node at (0,3)[above left] {1};
\node at (6,3)[above right] {1};
\node at (0,2)[left] {4};
\node at (6,2)[right] {4};
\node at (0,1)[left] {5};
\node at (6,1)[right] {5};
\node at (2,0)[below] {2};
\node at (2,3)[above] {2};
\node at (4,0)[below] {3};
\node at (4,3)[above] {3};
\node at (2,2)[below] {6};
\node at (4,1)[above] {7};

 \end{scope}
 \end{tikzpicture}
\caption{Left: the simplest triangulation of $\mathbb{S}^2$: the tetrahedron, 
inducing the complete graph $K_4$. Right: the $7$-point triangulation 
of the $2$-torus: the Császár polyhedron, inducing the complete graph $K_7$.}
\label{pyradjazenz}
\end{figure}

The Császár polyhedron with uniform side lengths is illustrated in Figure \ref{csaszar-uni},
where the colouring indicates which sides are glued together. 

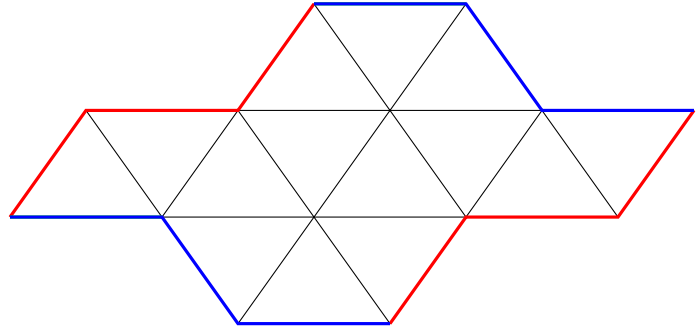
\begin{figure}[h]   
\begin{tikzpicture}
\draw (-1,2.12) -- (1,2.12);
\draw (-4,0.707) -- (4,0.707);
\draw (-5,-0.707) -- (3,-0.707);
\draw (-2,-2.12) -- (0,-2.12);

\draw (-2,0.707) -- (-1,2.12) -- (0,0.707) -- (1,2.12) -- (2,0.707);
\draw (-5,-0.707) -- (-4,0.707) -- (-3,-0.707) -- (-2,0.707) -- (-1,-0.707) -- (0,0.707) -- 
(1,-0.707) -- (2,0.707) -- (3,-0.707) -- (4,0.707);
\draw (-3,-0.707) -- (-2,-2.12) -- (-1,-0.707) -- (-0,-2.12) -- (1,-0.707);

\draw[red, very thick] (-5,-0.707) -- (-4,0.707) -- (-2,0.707) -- (-1,2.12);
\draw[red, very thick] (0,-2.12) -- (1,-0.707) -- (3,-0.707) -- (4,0.707);
\draw[green, very thick] (-5,-0.707) -- (-3,-0.707);
\draw[green, very thick] (-1,2.12) -- (1,2.12);
\draw[blue, very thick] (-5,-0.707) -- (-3,-0.707) -- (-2,-2.12) -- (0,-2.12);
\draw[blue, very thick] (-1,2.12) --(1,2.12) -- (2,0.707) -- (4,0.707);
\end{tikzpicture}
\caption{Császár polyhedron with uniform edge lengths.}
 \label{csaszar-uni}
\end{figure}

\subsection{Extremals of determinants for incomplete symmetric graphs}

Our second main result studies determinants on incomplete, but strongly symmetric graphs. 
The full statement is presented below in Theorem \ref{incomplete-th}
and reads informally as follows.

\begin{theorem}\label{main2}
Let $M = (V, E, T)$ be a triangulation of a smooth closed Riemannian surface, consisting of vertices $V$, edges $E$ and
triangle faces $T$. Consider a discrete metric $l: E \to (o,\infty)$ and impose the following conditions:
\begin{enumerate}
\item the graph is strongly symmetric in the sense of Definition \ref{strongly-symmetric-def},
or more generally satisfies the conditions laid out in Remark \ref{remark-constant}.
\item $M$ has no boundary, i.e. each edge is bounded by exactly two triangles.
\item the discrete metric $l$ induces angles $\pi/3$ in each triangle
\end{enumerate}
Then the determinant (excluding zero eigenvalues) 
of the cotan-Laplace operator $\Delta = \Delta_\text{cot}$ on $(M,l)$ is stationary 
with respect to variational changes of $l$. 
\end{theorem}

Examples that satisfy the conditions of the theorem include a $6$-point triangulation of 
the sphere $\mathbb{S}^2$, the octahedron, and the $9$-point triangulation of the $2$-torus. 
Both triangulations are illustrated in Figure \ref{torusbild}.

\begin{figure}[h]    
\centering
\begin{tikzpicture}[scale = 1]

\draw (0,0) circle (2);
\draw[dotted, thick] (0,0) ellipse (2 and 0.5);
\draw[very thick] (0,2) -- (-2,0) -- (0.5,-0.5) -- (2,0) --  (0,2) -- (-0.5,0.5);
\draw[very thick] (-2,0) -- (-0.5,0.5) -- (2,0);
\draw[very thick] (0,2) -- (0.5,-0.5);

\filldraw (0,2) circle (0.07);
\filldraw (0,-2) circle (0.07);
\filldraw (-2,0) circle (0.07);
\filldraw (-0.5,0.5) circle (0.07);
\filldraw (0.5,-0.5) circle (0.07);
\filldraw (2,0) circle (0.07);

\draw[very thick] (0,-2) -- (-2,0);
\draw[very thick] (0,-2) -- (2,0);
\draw[very thick] (0,-2) -- (0.5,-0.5);
\draw[very thick] (0,-2) -- (-0.5,0.5);

   \begin{scope}[shift={(6,0)}]
\draw (-2,-1) -- (-2,1) -- (2,1) -- (2,-1) -- (-2,-1);
\filldraw (-2,-1) circle (0.07);
\filldraw (-2,1) circle (0.07);
\filldraw (2,1) circle (0.07);
\filldraw (2,-1) circle (0.07);

\draw (-2,0.33) -- (2,0.33);
\filldraw (-2,0.33) circle (0.07);
\filldraw (2,0.33) circle (0.07);

\draw (-2,-0.33) -- (2,-0.33);
\filldraw (-2,-0.33) circle (0.07);
\filldraw (2,-0.33) circle (0.07);

\draw (-0.66,-1) -- (-0.66,1);
\filldraw (-0.66,-1) circle (0.07);
\filldraw (-0.66,1) circle (0.07);

\draw (0.66,-1) -- (0.66,1);
\filldraw (0.66,-1) circle (0.07);
\filldraw (0.66,1) circle (0.07);

\draw (-2,0.33) -- (-0.66,1);
\filldraw (-2,0.33) circle (0.07);
\filldraw (-0.66,1) circle (0.07);

\draw (-0.66,0.33) -- (0.66,1);
\filldraw (-0.66,0.33) circle (0.07);

\draw (0.66,0.33) -- (2,1);
\filldraw (0.66,0.33) circle (0.07);

\draw (-2,-0.33) -- (-0.66,0.33);
\draw (-0.66,-0.33) -- (0.66,0.33);
\draw (0.66,-0.33) -- (2,0.33);
\draw (-2,-1) -- (-0.66,-0.33);
\draw (-0.66,-1) -- (0.66,-0.33);
\draw (0.66,-1) -- (2,-0.33);
\node at (-2,-1) [below] {$1$};
\node at (-0.66,-1) [below] {$2$};
\node at (0.66,-1) [below] {$3$};
\node at (2,-1) [below] {$1$};
\node at (-2,1) [above] {$1$};
\node at (-0.66,1) [above] {$2$};
\node at (0.66,1) [above] {$3$};
\node at (2,1) [above] {$1$};
\node at (-2,1) [above] {$1$};
\node at (-2,0.33) [left] {$4$};
\node at (-2,-0.33) [left] {$7$};
\node at (2,0.33) [right] {$4$};
\node at (2,-0.33) [right] {$7$};
\node at (-0.66,0.33) [above left] {$5$};
\node at (0.66,0.33) [above left] {$6$};
\node at (-0.66,-0.33) [above left] {$8$};
\filldraw (-0.66,-0.33) circle (0.07);

\node at (0.66,-0.33) [above left] {$9$};
\filldraw (0.66,-0.33) circle (0.07);
 \end{scope}
 \end{tikzpicture}
\caption{Left: a $6$-point triangulation of $\mathbb{S}^2$, the octahedron. 
Right: the $9$-point triangulation of the $2$-torus.}
\label{torusbild}
\end{figure}
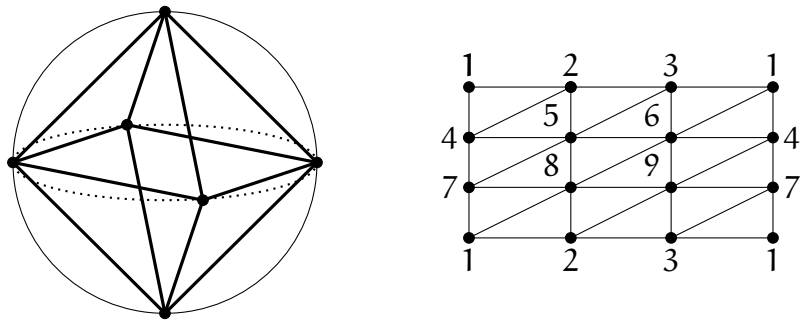

The $9$-point triangulation of the torus with uniform side lengths
is illustrated in Figure \ref{torus60bild}, where the colouring as before indicates 
which sides are glued together. Another example, where the conditions of Theorem \ref{main2}
apply, is the icosahedron. 

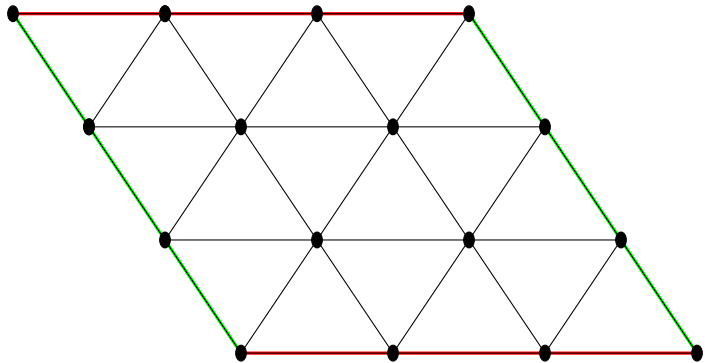
\begin{figure}[h] 
\centering
\begin{tikzpicture}[yscale=1.5]
\draw[red, very thick] (3,0) -- (9,0);
\draw[red, very thick] (6,3) -- (0,3);

\draw[green, very thick] (6,3) -- (9,0);
\draw[green, very thick] (0,3) -- (3,0);

\draw (3,0) -- (9,0) -- (6,3) -- (0,3) -- (3,0);
\filldraw (3,0) circle (0.07);
\filldraw (9,0) circle (0.07);
\filldraw (6,3) circle (0.07);
\filldraw (0,3) circle (0.07);

\draw (3,0) -- (4,1);
\filldraw (4,1) circle (0.07);

\draw (5,0) -- (6,1);
\filldraw (5,0) circle (0.07);
\filldraw (6,1) circle (0.07);

\draw (7,0) -- (8,1);
\filldraw (7,0) circle (0.07);
\filldraw (8,1) circle (0.07);

\draw (5,0) -- (2,3);
\filldraw (2,3) circle (0.07);

\draw (7,0) -- (4,3);
\filldraw (4,3) circle (0.07);

\draw (1,2) -- (7,2);
\filldraw (1,2) circle (0.07);
\filldraw (7,2) circle (0.07);

\draw (1,2) -- (2,3);
\filldraw (1,2) circle (0.07);
\filldraw (2,3) circle (0.07);

\draw (3,2) -- (4,3);
\filldraw (3,2) circle (0.07);

\draw (5,2) -- (6,3);
\filldraw (5,2) circle (0.07);

\draw (2,1) -- (8,1);
\filldraw (2,1) circle (0.07);

\draw (2,1) -- (3,2);
\draw (4,1) -- (5,2);
\draw (6,1) -- (7,2);
\end{tikzpicture}
\caption{The $9$-point triangulation of $\mathbb{T}^2$ 
with uniform edge lengths.}
\label{torus60bild}
\end{figure}

\subsection{Extremals of determinants are local minima}\label{minima-section}

Our final main observation consists of numerical illustrations which 
indicate that the stationary triangulations found
above may be (local) \textbf{minima}. This does not fit into the spirit of Theorem \ref{OPS}, where the
constant curvature metrics are actually global maxima. This shows that some crucially different phenomena 
may arise in the discrete setting versus the smooth case. 

\begin{explanation}
Although an analogy with the Osgood-Phillips-Sarnak theorem is desirable, 
the authors see no mathematical reason why this can be 
hoped for, even if the discretization is refined sufficiently. Indeed, \cite{Iz} establishes a quite 
general asymptotic behaviour of discrete determinants (though for a different discrete Laplacian)
under specific refinements of the discretization. The constant term in the asymptotics is 
the logarithm of the zeta-regularized determinant and varies with the metric according to 
the Osgood-Phillips-Sarnak theorem. However, the logarithmic term in the expansion depends non-trivially
on the angles of the various triangles in the triangulation and does not satisfy a variational 
formula in any obvious way. 
\end{explanation}

\noindent \textbf{\emph{Acknowledgements:}} The authors thank Daniel Grieser for his invaluable comments
and a careful reading of the manuscript. Furthermore we want to gratefully acknowledge Daniel 
Grieser's observation of an equivalent characterization of strongly symmetric graphs, which we wrote down in
Proposition \ref{path-characterization}.

\section{Preliminary elements of discrete geometry}

In this section we shall recall some fundamental notions from discrete geometry
that we are using throughout this note. Consider a closed Riemann surface 
and its triangulation $(V, E, T)$ as e.g. in \cite[Theorem 5.36]{Lee}, 
where $V$ denotes the set of vertices, $E$ the
set of edges and $T$ the set of triangles. The triangulation has \emph{no boundary}, 
i.e. each edge appears in the boundary of exactly two
triangles.

\begin{definition}
A \emph{discrete metric} on a triangulation $(V, E, T)$ is a collection of lengths for each edge, i.e. 
a function ${l \colon E \to (0, \infty)}$ such that for each triangle $(i,j,k) \in T$ with vertices $i,j,k \in V$ and 
edges $(i,j), (i,k), (j,k) \in E$ the lengths $l_{vw} = l(v,w)$ satisfy the triangle inequalities
\begin{align*}
l_{ij} &< l_{jk} + l_{ki}, \\
l_{jk} &< l_{ki} + l_{ij}, \\
l_{ki} &< l_{ij} + l_{jk}.
\end{align*}
\end{definition}

Lengths of the three sides in an Euclidean triangle determine the angles and the 
area uniquely. This way, the length functional, i.e. the discrete metric, defines angles and area of the individual triangles of the 
triangulation and hence paves the way for the next definition.

\begin{definition} Consider a triangulation $(V, E, T)$.
Let $T_i \subset T$ denote the set of triangles with $i \in V$ as a vertex. Denote by 
$\alpha_{i}(t)$ the angle in $t \in T_i$ at that vertex. Define the \emph{discrete Gaussian curvature} $K_i$ at 
the vertex $i$ as the angle defect
\[K_i = \pi - \sum_{t \in T_i} \alpha_{i}(t).\]
\end{definition}
In virtue of a discrete metric and its induced angles, we can furthermore define a discrete Laplace operator 
whose weights are defined in terms of these angles.

\begin{definition}\label{cotan-def}
Given a discrete metric on a triangulation without boundary, the discrete \emph{cotan-Laplace operator} $\Delta_\text{cot}$ is a linear operator acting on functions 
$f: V \to \mathbb{R}$ as follows. For any vertex $i \in V$ and emanating edge $(i,j) \in E$ consider 
the two angles $\alpha_{ij}$ and $\beta_{ij}$ opposite to the edge $(i,j)$ as in Figure \ref{cotanerklaer}. 
Let $A_i$ denote the so-called "local area element", 
given by one third of the total area of all triangles with a vertex at $i \in V$. 
We define the edge weights 
$$
w_{ij} = \cot \alpha_{ij} + \cot \beta_{ij}.
$$ 
\begin{enumerate}
\item Then the cotan Laplacian is defined by
$$
(\Delta_\text{cot} f) (i) := \sum_{(i,j) \in E} w_{ij} (f(i) - f(j)).
$$
\item The \emph{"local area normalized"} cotan Laplacian is defined by
$$
(\Delta'_\text{cot} f) (i) := A_i^{-1} \sum_{(i,j) \in E} w_{ij} (f(i) - f(j)).
$$
\end{enumerate}
\end{definition}

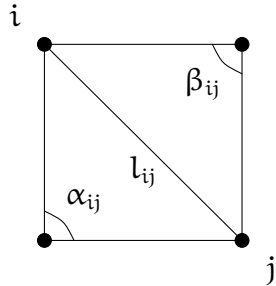
\begin{figure}[h]
\centering
\begin{tikzpicture}[scale=1.3]
\draw (0,0) -- (2,0) -- (2,2) -- (0,2) -- (0,0);
\draw (2,0) -- (0,2);
\filldraw (0,0) circle (0.07);
\filldraw (2,0) circle (0.07);
\filldraw (2,2) circle (0.07);
\filldraw (0,2) circle (0.07);

\draw (0.3,0) .. controls (0.2,0.2) and (0.2,0.2) .. (0,0.3);
\draw (1.7,2) .. controls (1.8,1.8) and (1.8,1.8) .. (2,1.7);

\node at (1,0.7) {$l_{ij}$};
\node at (0.4,0.4) {$\alpha_{ij}$};
\node at (1.6,1.6) {$\beta_{ij}$};
\node at (2.3,-0.3) {$j$};
\node at (-0.3,2.3) {$i$};

\end{tikzpicture}
\caption{$\alpha_{ij}$ and $\beta_{ij}$ in the formula for the cotan-Laplace operator. \label{cotanerklaer}}
\end{figure}

One can enumerate the vertices of a graph, so that a function $f \colon V \to \R$ is represented 
by a vector which we also call $f$. The action of $\Delta_\text{cot}$ is then given by the matrix 
(also denoted by $\Delta_\text{cot}$) with entries ( $w_{ij} := 0$ if $(i,j) \notin E$ )
\begin{align}\label{matrix-cotan}
\begin{cases}
&(\Delta_\text{cot})_{ii} = \sum\limits_{(i,j) \in E} w_{ij}, \qquad \quad 
(\Delta'_\text{cot})_{ii} = \sum\limits_{(i,j) \in E} w_{ij} A^{-1}_i, \\
&(\Delta_\text{cot})_{ij} = - w_{ij}, \ \text{if $i \neq j$}, \quad
(\Delta'_\text{cot})_{ij} = - w_{ij} A^{-1}_i, \ \text{if $i \neq j$}.
\end{cases}
\end{align}

\begin{definition}
We define the determinants $\det' \Delta_\text{cot}$ and $\det' \Delta'_\text{cot}$ 
as the product of the nonzero eigenvalues of the respective matrices in \eqref{matrix-cotan}. 
This is well-defined, since the eigenvalues of \eqref{matrix-cotan} 
do not depend on the order in which the 
vertices are enumerated.
\end{definition}

The cotan-Laplacian is widely considered to be a 'good' choice of a geometric discrete Laplace 
operator that can approximate the classic Laplace-Beltrami operator on a Riemann surface. 
Moreover, $\Delta_\text{cot}$ appears in the evolution equation for the 
curvature $\{K_i\}$ along discrete Ricci flow, as noted in \cite{confequiv}. 
However, we remark that there are various other choices with different weights $w_{ij}$.

\section{Auxiliary results on matrices, their inverses and determinants}

\subsection{An auxiliary result on the derivative of determinants}

In order to calculate the derivative of the discrete determinant in the proofs our two main 
results, Theorems \ref{main1} and \ref{main2} (corresponding to Theorems \ref{sphth} and \ref{incomplete-th} below), 
we will use the following well-known formula.

\begin{lemma} \label{logdet}
For any differentiable family $(L(t))_{t \in \R} \subset \R^{n \times n}$ of positive semi-definite symmetric matrices  
of constant rank, the derivative of its determinant is given by
\[  \frac{d}{dt}  \log \Det L(t) = \lim_{\delta \to 0} \tr(\dot{L}(t) (L(t) + \delta)^{-1}).\]
\end{lemma}
\begin{proof}
For any differentiable one-parameter families of matrices $B(t), C(t) \in \R^{n \times n}$ and $M(t) := B(t) C(t)$ it holds that
\begin{align}\label{eq1}
\tr(\dot{M} M^{-1}) = \tr((\dot{B} C + B \dot{C}) C^{-1} B^{-1}) = \tr(\dot{B} B^{-1}) + \tr(\dot{C} C^{-1}),
\end{align}
where e.g. $\dot{M}$ denotes differentiation of $M(t)$ in $t$. Thus, if $B = B(t)$ 
is a family of invertible matrices and $C = B(t)^{-1}$, 
the following expression vanishes (here $M(t) := B(t)B^{-1}(t)\equiv I$ with $I$ being the identity matrix):
\begin{align}\label{eq2}
\tr(\dot{B} B^{-1}) + \tr \left( B \left( \frac{d}{dt} B^{-1} \right)\right) = \tr(\dot{I}I^{-1}) = 0.
\end{align}
For diagonalizable matrices $M(t)$, written as $M(t)= Q^{-1}(t) D(t) Q(t)$ for a diagonal matrix family $D(t)$, 
we obtain using \eqref{eq1} twice and \eqref{eq2} once
\begin{align}\label{eq3}
\tr(\dot{M} (M^{-1})) = \tr(\dot{D} D^{-1}) +  \tr \left( \left( \frac{d}{dt} Q^{-1} \right) Q\right) + \tr(\dot{Q} Q^{-1}) = \tr(\dot{D} D^{-1}).
\end{align}
We now consider $M(t) := L(t)$ with its diagonalization 
$$
D(t) := \textup{diag} (\lambda_1(t), \dots, \lambda_n(t)),
$$
with the diagonal of eigenvalues of $A(t)$. We shall enumerate $\lambda_m(t), \dots, \lambda_n(t)$ as the positive eigenvalues
with $m-1 = \dim \ker L(t)$ independent of $t$ by assumption. Note that from $\lambda_1(t), \dots, \lambda_{m-1}(t) \equiv 0$ it follows that $\dot \lambda_1, \dots, \dot \lambda_{m-1} = 0$. We compute using \eqref{eq3}
\begin{align*}
\lim \limits_{\delta \to 0} \tr(\dot{L} (L + \delta)^{-1}) &= \lim \limits_{\delta \to 0} \tr(\dot{D} (D + \delta)^{-1}) = \lim \limits_{\delta \to 0}  \sum_{i=0}^n \frac{\dot \lambda_i}{\lambda_i + \delta} \\ &= \lim \limits_{\delta \to 0} \sum_{i=m}^n \frac{\dot \lambda_i}{\lambda_i + \delta} + \frac{0}{\delta} = \sum_{i=m}^n \frac{\dot \lambda_i}{\lambda_i} \\
&= \sum_{i=m}^n  \frac{d}{dt}  \log(\lambda_i) =  \frac{d}{dt}  \log\left(\prod_{i=m}^n \lambda_i\right) \\
&=  \frac{d}{dt} \log \Det L.
\end{align*}
\end{proof}

\subsection{An auxiliary result on traces of matrix products}

In view of Lemma \ref{logdet}, we need auxiliary results that will allow us to 
evaluate the trace of products of the form $\dot{L}(t) (L(t) + \delta)^{-1}$.
The first such auxiliary result deals with the trace of products of certain matrices.

\begin{lemma} \label{matlemma}
Let $A, B \in \R^{n \times n}$ be symmetric matrices, such that
\begin{enumerate}
\item $A= (a_{ij})$ is trace-free with $\sum_{i \neq j} a_{ij} = 0$,
\item $B= (b_{ij})$ satisfies $b_{ii} = x$, and $b_{ij} = y$, $i\neq j$, whenever $a_{ij} \neq 0$.
\end{enumerate} 
Then $AB$ is also trace-free.
\end{lemma}
\begin{proof}
The claim follows by a simple calculation
\begin{align*}
\tr(AB) &= \sum_i \sum_j a_{ij} b_{ji} = \sum_i \left(a_{ii} b_{ii} + \sum_{j, \: j \neq i} a_{ij} b_{ji}\right)  \\
&= x \sum_i a_{ii} + y \sum_{i \neq j} a_{ij} = 0.
\end{align*}
\end{proof}

\subsection{Explicit structure of some inverse matrices}

The next result is needed to check if $AB:=\dot{L}(t) (L(t) + \delta)^{-1}$, with $L(t)$ being the family 
of cotan-Laplacians, has the structure required in Lemma \ref{matlemma} in case of complete graphs, as needed in
Theorem \ref{sphth}.

\begin{lemma} \label{matlemma1}
Let $B= (b_{ij})\in \R^{n \times n}$ be a symmetric matrix with $b_{ii} = x$, and $b_{ij} = y$, if $i\neq j$.
Assume that $(x-y)(x+(n-1)y) \neq 0$. Then $B^{-1} = (c_{ij})$ exists with
\[
c_{ii} = \frac{x + (n-2) y}{(x-y)(x+(n-1)y)}, \ c_{ij} = -\frac{y}{(x-y)(x+(n-1)y)} \ \text{for} \ i \neq j,
\]
\end{lemma}

\begin{proof}
Consider the matrix $C' = (c'_{ij}) \in \R^{n \times n}$ with entries 
\[
c'_{ii} = x + (n-2) y, \ c'_{ij} = -y \ \text{for} \ i \neq j,
\]
and compute
\begin{align*}
(BC')_{ii} &= \sum_j b_{ij} c'_{ji} = b_{ii} c'_{ii} + \sum_{j, \: j \neq i} c'_{ij} b_{ji}
\\ &= x (x + (n-2) y) - (n-1) y^2 = (x-y) (x+ (n-2) y), \\
(BC')_{ij} &= \sum_k b_{ik} c'_{kj} =  b_{ii} c'_{ij} + b_{ij} c'_{jj} + \sum_{k \neq i, j} b_{ik} c'_{kj}
\\ &= -xy + y (x + (n-2) y) - (n-2) y^2 = 0 \ \text{for} \ i \neq j.
\end{align*}
We conclude that the inverse $B^{-1}$ exists with
$$
B^{-1} = \frac{C'}{(x-y)(x+(n-1)y)}.
$$ 
\end{proof}

\section{Properties of strongly symmetric graphs}

The content of this section is needed to check if $AB:=\dot{L}(t) (L(t) + \delta)^{-1}$,
with $L(t)$ being the family of cotan-Laplacians, has the 
structure required in Lemma \ref{matlemma} in case of graphs that are not necessarily complete.
Instead of completeness, we introduce some symmetry conditions that will be assumed in 
Theorem \ref{incomplete-th} below. We shall begin 
with a definition of such symmetry conditions here.

\begin{definition}\label{strongly-symmetric-def} 
Consider a finite graph\footnote{ This graph need not arise from a 
triangulation of a smooth surface at the moment. } $(V,E)$ consisting of vertices $V$ and edges $E \neq \varnothing$.
We assume that the graph is connected, i.e. for any two vertices $i,j \in V$ there
exists a path of edges, connecting $i$ and $j$. We say that two vertices are "neighbors", if they are connected by
an edge. We introduce some notation.
\begin{enumerate}
\item $L(i,j)$ is the number of edges in a shortest path from $i$ to $j$ (the "distance"),
\item $\alpha_s(i,j)$ is the number of neighbors $k$ of $i$ that have distance $s$ from $j$.
\end{enumerate}
We call a graph \emph{strongly symmetric} if for any $s \in \N_0$ 
$$L(i,j) = L(v,w) \Rightarrow \alpha_s(i,j) = \alpha_s(v,w).$$ 
We then write $\alpha_s(\ell) := \alpha_s(i,j)$ for $\ell = L(i,j)$.
\end{definition}

Note that $L(i,j) = 0$ iff $i = j$ and $L(i,j) = 1$ iff $(i,j) \in E$. 
For given $i,  j \in V$ with $L(i,j) = \ell$, the longest possible distance of a vertex $k \in V$ with $(i,k) \in E$ to $j$ is $L(k,j) = \ell + 1$, 
since one can always choose the $1$-edge path from $k$ back to $i$ followed by a shortest path from $i$ to $j$. On the other hand, 
the shortest possible distance from such $k$ to $j$ is $L(k,j) = \ell-1$ because if $L(k,j) \leq \ell - 2$, this would contradict $L(i,j) = \ell$ 
by the same reasoning. Hence for any $\ell$ 
\begin{equation} \label{alpha-zero}
\alpha_s(\ell) \neq 0 \ \Rightarrow \ s \in \{\ell-1, \ell, \ell+1\}.
\end{equation}
Strong symmetry moreover implies that the number of edges, emanating from each $i \in V$, i.e. 
the degree of a vertex, is the same for each $i$
and is given for each $\ell$ by 
\begin{equation}
\alpha_1(i,i) \equiv \alpha_1(0) = \alpha_{\ell}(\ell) + \alpha_{\ell-1}(\ell) + \alpha_{\ell+1}(\ell).
\end{equation}
Finally, for a finite graph we may define $L:= \max \{ L(i,j) \mid i,j \in V\}$ and find
\begin{equation}\label{alpha-L}
\alpha_{L+1} (L) = 0, \quad
\alpha_L(L) = \alpha_1(0) - \alpha_{L-1}(L).
\end{equation}

Examples of a strongly symmetric graph are the octahedron
and the $9$-point triangulation of the $2$-torus as both shown in Figure \ref{torusbild}. 
It will be of advantage to employ an equivalent characterization of 
strongly symmetric graphs.

\begin{proposition}\label{path-characterization}
A finite connected graph $(V,E)$ is strongly symmetric if and only if
the number $p_m(i,j)$ of (obviously not necessarily shortest) paths of length $m \in \N$ between vertices $i,j \in V$
depends only on the distance $\ell = L(i,j)$ between $i$ and $j$. In such a case we write
$$
p_m(i,j) = p_m(\ell).
$$
\end{proposition}

\begin{proof}
Assume that $(V,E)$ is strongly symmetric. Obviously,
$p_1(i,j) = 1$ if $i$ and $j$ are neighbors, i.e. connected by an edge, and zero otherwise. 
Hence $p_1(i,j)$ in fact depends only on the distance $\ell = L(i,j)$.
Assume by induction that $p_{m-1}(i,j)$ depends only on the distance $\ell = L(i,j)$
for some $m \in \N$. Then for any $i,j \in V$ with $\ell = L(i,j)$ we find
\begin{align*}
p_m(i,j) &= \sum_{s}  \alpha_{s}(\ell) p_{m-1}(s)
\\ &=  \alpha_{\ell-1}(\ell) p_{m-1}(\ell-1) + \alpha_{\ell}(\ell) p_{m-1}(\ell) + \alpha_{\ell+1}(\ell)p_{m-1}(\ell+1).
\end{align*}
By this formula, $p_m(i,j)$ also depends only on the distance $\ell = L(i,j)$.
Hence this is true for all $m$. Conversely, assume that $p_m(i,j)$ depends only on the distance $\ell = L(i,j)$.
We shall show that the graph is strongly symmetric. For this we note 
\begin{align*}
&p_\ell(\ell) = \alpha_{\ell-1}(i,j) p_{\ell-1}(\ell-1), \\
&p_{\ell+1}(\ell) = \alpha_{\ell-1}(i,j) p_{\ell}(\ell-1) + \alpha_{\ell}(i,j) p_{\ell}(\ell), \\
&p_{\ell+2}(\ell) = \alpha_{\ell-1}(i,j) p_{\ell+1}(\ell-1) + \alpha_{\ell}(i,j) p_{\ell+1}(\ell)
 + \alpha_{\ell+1}(i,j)p_{\ell +1}(\ell+1).
\end{align*}
From here we conclude that $\alpha_{\ell \pm 1}(i,j)$ and $\alpha_{\ell}(i,j)$ indeed depend
only on $\ell = L(i,j)$ and hence the graph is strongly symmetric.
\end{proof}

Proposition \ref{path-characterization} does not provide a feasible way to check if a 
given graph is actually strongly symmetric, since in general we cannot study the number $p_m(i,j)$ of paths 
for \emph{arbitrary} length $m \in \N$. However, Proposition \ref{path-characterization} allows
for a simple proof of the following lemma, which is an analogue of Lemma \ref{matlemma1} and is key
for the proof of Theorem \ref{incomplete-th}.

\begin{lemma} \label{incomplete-lemma2}
Let $G = (V, E)$ be a strongly symmetric graph. 
Enumerate the vertices $V = \{1, \dots, n\}$ and consider for any $x,y \in \C$ 
a matrix $B= (b_{ij}) \in \C^{n \times n}$ with
\begin{align*}
\begin{cases}
b_{ii} = x, \: &i \in V, \\
b_{ij} = y, \: &(i,j) \in E, \\
b_{ij} = 0, \: &(i,j) \notin E, \, i \neq j.
\end{cases}
\end{align*}
Then, for any given $y$, $B$ is invertible for almost all $x$ 
and the inverse $B^{-1} = (c_{ij})$ satisfies: 
$c_{ij}$ depends only on $L(i,j)$ and $x,y$. In particular, there exist real-valued functions 
$x_0(x,y)$ and $x_1(x,y)$ such that 
$$ 
\forall i \in V: c_{ii} =x_0(x,y), \quad \forall (i,j) \in E: c_{ij} = x_1(x,y). 
$$ 
\end{lemma}
\begin{proof}
We can write the matrix $B=B(x,y)$ as $B(x,y)=x I +y A$, where $I$ is the identity 
$n \times n$ matrix and $A = (a_{ij})$ with $a_{ij} = 1$ if $(i,j) \in E$ and zero otherwise.
Using the cofactor formula, $B(x,y)^{-1}$ is a meromorphic function in $x$, for a fixed $y$.
For $|x| \geq |y|$ we can write using Neumann series
$$
B(x,y)^{-1} = x^{-1} \sum_{m=0}^\infty (-1)^m \left( \frac{y}{x}\right)^m A^m.
$$
Note that the individual entries of $A^m = (a^m_{ij})$ can be written as
$$
a^m_{ij} = \sum_{i_1,...,i_{m-1}} a_{ii_1} a_{i_1i_2} \cdots a_{i_{m-1}j}.
$$
Since $a_{ij} = 1$ if $(i,j) \in E$ and zero otherwise, $a^m_{ij} = p_m(i,j)$ is simply the
number of paths of length $m$ from $i$ to $j$, in the notation of Proposition \ref{path-characterization}. 
Consequently the value of the individual entries of $B(x,y)^{-1} = (c_{ij}(x,y))$ with $L(i,j) = \ell$
fixed, depends only on $(x,y)$ and $\ell$. In particular, given any vertex pairs $(i,j)$ and $(i',j')$
with $L(i,j) = L(i',j')$, we conclude that $c_{ij}(x,y) = c_{i'j'}(x,y)$ for $|x| \geq |y|$.
Hence $c_{ij}(\cdot,y) \equiv c_{i'j'}(\cdot,y)$ 
as meromorphic functions. This proves the statement. 
\end{proof}

We close the section with an alternative proof of the same statement, without
using Proposition \ref{path-characterization}. 

\begin{proof}[Alternative proof of Lemma \ref{incomplete-lemma2}]
We construct an inverse of $B$ by making the ansatz that $B^{-1} =(c_{ij})$ has entries 
$c_{ij} = x_\ell \equiv x_\ell(x,y)$, where $\ell = L(i,j)$. Considering  the diagonal terms, $BC = \Id$ yields the equation 
\begin{equation} \label{recur-1}
(BC)_{ii} = b_{ii} c_{ii} + \sum_{i \neq k} b_{ik} c_{ki} = x x_0 + \alpha_1(0) y x_1 = 1.
\end{equation}
For the non-diagonal terms, consider first $(i,j) \in E$. Then $L(i,j) = 1$ and
\begin{equation} \label{recur-2}
(BC)_{ij} = b_{ii} c_{ij} + b_{ij} c_{jj} + \sum_{k \neq i,j} b_{ik} c_{kj} = x x_1 + y x_0 + \alpha_1(1) y x_1 + \alpha_2(1) y x_2 = 0.
\end{equation}
For general $i,j \in V$ with $L(i,j) = \ell \in [2,L-1]$ we obtain using \eqref{alpha-zero}
\begin{equation} \label{recur-3}
\begin{split}
(BC)_{ij} &= b_{ii} c_{ij} + b_{ij} c_{jj} + \sum_{k \neq i,j} b_{ik} c_{kj} \\ &= x x_\ell + 
\alpha_{\ell-1}(\ell) y x_{\ell-1} + \alpha_{\ell}(\ell) y x_{\ell} + \alpha_{\ell+1}(\ell) y x_{\ell+1} = 0.
\end{split}
\end{equation}
Let now $L(i,j) = L$, i.e. $i$ and $j$ are two vertices of maximal distance between each other. In this case we obtain
using \eqref{alpha-L}
\begin{equation} \label{recur-L}
(BC)_{ij} = x x_L +  \alpha_{L-1}(L) y x_{L-1} +  \alpha_{L}(L) y x_{L} = 0.
\end{equation}
We shall solve the system \eqref{recur-1}, \eqref{recur-2}, \eqref{recur-3} and 
\eqref{recur-L} recursively. Note that the number $\alpha_1(0)$ of emanating edges 
from each vertex in a strongly symmetric graph is positive, since we assume 
$E \neq \varnothing$. Moreover $a_{\ell+1}(\ell) \neq 0$ for 
$\ell < L$. Indeed, given $i,j \in V$ with maximal distance $L(i,j) = L$,
consider the corresponding path of edges. Enumerate the vertices along that path by
$i = i_0, \cdots, i_L=j$. Then $L(i_{L-\ell},i_L)= \ell$ and $L(i_{L-\ell-1},i_L)= \ell+1$. Thus, by definition, 
$a_{\ell+1}(\ell) \neq 0$ for $\ell < L$ and $a_{L+1}(L) = 0$. We then obtain (assume without loss of generality that
$y \neq 0$, since for $y=0$ the matix $B=xI$ and the statement is trivial)
\begin{align}
\mbox{\eqref{recur-1}} \, \Rightarrow \, &x_1 = \frac{1 - x x_0}{\alpha_1(0)y }, \nonumber \\
\mbox{\eqref{recur-3}} \, \Rightarrow \, &x_{\ell + 1} = - \frac{x x_\ell + 
\alpha_{\ell-1}(\ell) y x_{\ell-1} + \alpha_{\ell}(\ell) y x_{\ell}}{\alpha_{\ell+1}(\ell) y}, \, 
2 \leq \ell \leq L-1, \nonumber \\
\mbox{\eqref{recur-3}} \, \Rightarrow \, &x_{L} = - \frac{\alpha_{L-1}(L) y x_{L-1}}{ \alpha_{L}(L) y + x}. \label{recur}
\end{align}
While the first two lines in \eqref{recur} above 
can be used to express $x_1, \cdots, x_{L}$ in terms of $x_0$,
the third line allows us to compute $x_0$ in terms of $x$ and $y$, thereby giving 
an explicit solution to all $x_0, \cdots, x_{L}$. Let us be precise. 
The recursion, i.e. the second equation in \eqref{recur}, yields (we shall consider $y \neq 0$ fixed and make the dependence
on $x$ explicit)
\[
x_{\ell} = p'_{\ell}(x) x_0 + p''_{\ell-1}(x),
\]
where $p'_\ell(x)$ and $p''_\ell(x)$ are polynomials in $x$ of degree $\ell$.
The leading coefficient of $p'_{\ell}(x)$ is non-zero for $\ell \leq L-1$ and is given explicitly by
\begin{align}\label{leading-coeff}
(-1)^\ell y^{-\ell} \left( \prod\limits_{k=0}^\ell \alpha_{k+1}(k) \right)^{-1}
\end{align}
Now, plugging this into the last equation in \eqref{recur} yields
\begin{align*}
&p'_{L}(x) x_0 + p''_{L-1}(x) = 
- \frac{\alpha_{L-1}(L) y}{ \alpha_{L}(L) y + x}
\Bigl( p'_{L-1}(x) x_0 + p''_{L-2}(x) \Bigr)  \\
&\Rightarrow x_0  = - \frac{ p''_{L-1}(x)\bigl(\alpha_{L}(L) y + x\bigr) + 
\alpha_{L-1}(L) y p''_{L-2}(x)}{p'_{L}(x) 
\bigl(\alpha_{L}(L) y + x\bigr)+ 
\alpha_{L-1}(L) y p'_{L-1}(x)} =: \frac{q_{L}(x)}{q_{L+1}(x)}.
\end{align*}
The crucial aspect to note is that the denominator in the formula for $x_0$
is a polynomial of degree $(L+1)$ with non-zero leading coefficient \eqref{leading-coeff}.
Hence the formula for $x_0$ makes sense, and the system of equations \eqref{recur}
can be solved uniquely for almost all values of $x$, up to the finitely many zeros of
the polynomial $q_{L+1}$.
\end{proof}

\section{Stationary point of determinants for complete graphs}

We can now prove our first main result.

\begin{theorem} \label{sphth}
Let $M = (V, E, T)$ be a triangulation such that 
\begin{enumerate}
\item the graph is complete, i.e. all vertices are connected by an edge.
\item $M$ has no boundary, i.e. each edge is bounded by exactly two triangles.
\item the discrete metric $l$ induces angles $\pi/3$ in each triangle.
\end{enumerate}
Consider a smooth family of discrete metrics $l(\varepsilon)$,
such that for $\varepsilon = 0$ all angles in the triangles are equal to $\pi/3$. Consider the 
corresponding family of cotan-Laplace operators $\Delta_\varepsilon = \Delta_\text{cot} (l(\varepsilon))$ 
and their area normalized versions $\Delta'_\varepsilon$ 
as in Definition \ref{cotan-def}. 
\begin{enumerate}
\item[(i)] Then $\Delta_\varepsilon$ is stationary at $\varepsilon = 0$, i.e.
$$
\left. \frac{d}{d\varepsilon} \right|_{\varepsilon = 0} \log \Det \Delta_\varepsilon = 0.
$$
\item[(ii)] Assume the total area of triangles for $l(\varepsilon)$ is fixed. 
Then 
$$
\left. \frac{d}{d\varepsilon} \right|_{\varepsilon = 0} \log \Det \Delta'_\varepsilon = 0.
$$
\end{enumerate}
\end{theorem}

\begin{proof} 
The discrete metric $l(\varepsilon)$ defines for each triangle $t \in T$ side lengths and hence also
angles, which we denote by $(\alpha(\varepsilon), \beta(\varepsilon), \gamma(\varepsilon))$.
Since the angles for $\varepsilon = 0$ are all given by $\pi/3$, we can expand
\begin{align}\label{alpha-exp}
	(\alpha(\varepsilon), \beta(\varepsilon), \gamma(\varepsilon)) = (\pi/3, \pi/3, \pi/3) + 
	\varepsilon (t_\alpha, t_\beta, t_\gamma) + O(\varepsilon^2), \ \textup{as} \ \varepsilon \to 0,
\end{align}
where $(t_\alpha, t_\beta, t_\gamma)$ are the infinitesimal angle changes. 
Because the sum of angles in each triangle is $\pi$, independently of $\varepsilon$, we conclude that
\begin{equation}
t_\alpha + t_\beta + t_\gamma = \left. \frac{d}{d\varepsilon} \right|_{\varepsilon = 0} 
\bigl( \alpha(\varepsilon) + \beta(\varepsilon) + \gamma(\varepsilon) \bigr)
= 0. \label{tsum0}
\end{equation}
Let us now consider the family of cotan-Laplace operators $\Delta_\varepsilon$, identified with
their corresponding matrices \eqref{matrix-cotan}. Expanding its weights up to linear order, we get
in view of \eqref{alpha-exp}
\begin{equation}\label{w-expansion}
w_{ij}(\varepsilon) = \cot \alpha_{ij}(\varepsilon) + 
\cot \beta_{ij}(\varepsilon) = \frac{2}{\sqrt{3}} - \frac{4}{3} (t_{\alpha_{ij}} + 
t_{\beta_{ij}}) \varepsilon + \oo(\varepsilon^2).
\end{equation}
Let us abbreviate 
$$
\dot{\Delta} := \left. \frac{d}{d\varepsilon} \right|_{0} \Delta_\varepsilon, \quad  
\dot{w}_{ij}:= \left. \frac{d}{d\varepsilon} \right|_{0} w_{ij}(\varepsilon). 
$$ 
\underline{\emph{Proof of statement $(i)$:}}
Observe that by the formula \eqref{matrix-cotan} and the equation \eqref{tsum0}
\begin{equation}\label{trzero}
\begin{split}
\tr \dot{\Delta}
= \sum_{i \in V} (\dot{\Delta})_{ii} = 
\sum_{i \in V} \sum\limits_{(i,j) \in E} \dot{w}_{ij} 
= 2 \sum\limits_{(i,j) \in E} \dot{w}_{ij} = 
- \frac{8}{3} \sum\limits_{t \in T} (t_\alpha + t_\beta + t_\gamma) = 0.
\end{split}
\end{equation}
Similarly, we compute for the sum of the off-diagonal terms in the matrix \eqref{matrix-cotan}
\begin{align}\label{offdiagonal}
\sum_{i \neq j} (\dot{\Delta})_{ij} = - 2 \sum\limits_{(i,j) \in E} \dot{w}_{ij} = 0.
\end{align}
One easily sees (here it is essential that the graph is complete, i.e. all vertices are connected and hence
the number of edges emanating from each vertex is the same) 
that $\Delta_0$ and $(\Delta_0 + \delta)$ for any $\delta > 0$
are of the same structure as $B$ in Lemma \ref{matlemma1}. Noting that 
$(\Delta_0 + \delta)$ is invertible for $\delta > 0$,
Lemma \ref{matlemma1} asserts that for $(\Delta_0 + \delta)^{-1}$
all entries on the diagonal are equal, and all entries
away from the diagonal are equal. Hence by Lemma \ref{matlemma} 
\begin{align*}
\tr \dot{\Delta} (\Delta_0 + \delta)^{-1} \equiv 0.
\end{align*}
Here we should point out, that as long as all weights $w_{ij}> 0$ (which is the case for $\varepsilon$ sufficiently small
due to $w_{ij}(0) = \frac{2}{\sqrt{3}}$) the kernel of 
$\Delta_\varepsilon$ equals the number of connected components of the graph, i.e. $1$ in this case. 
Thus, each $\Delta_\varepsilon$ has the same rank and we obtain by Lemma \ref{logdet}
\begin{align*}
\left. \frac{d}{d\varepsilon} \right|_{0} \log \Det \Delta_\varepsilon  
= \lim_{\delta \to 0} \tr \dot{\Delta}_0 (\Delta_0 + \delta)^{-1} = 0.
\end{align*}
\underline{\emph{Proof of statement $(ii)$:}}
We write $A_i(\varepsilon)$ for the local area elements (given by one third of the area of all
triangles, cf. Definition \ref{cotan-def}) at the vertex $i$ for the metric $l(\varepsilon)$. 
Note that $A_i(0) = A$ is constant, since for $l(\varepsilon=0)$, 
all triangles have equal sides and by assumption, the number of edges emanating from each vertex
is the same (hence so is the number of triangles bordering each vertex). Hence we may expand
\begin{align*}
A_i(\varepsilon) &= A + \varepsilon a_i + O(\varepsilon^2),  \\
\frac{1}{A_i(\varepsilon)} &= \frac{1}{A} - \frac{\varepsilon a_i}{A^2} + 
O(\varepsilon^2), \ \textup{as} \ \varepsilon \to 0.
\end{align*}
Since the total area of the triangles is constant in $\varepsilon$, we find 
\begin{align}\label{asum}
\sum_{i \in V} a_i  = 0.
\end{align}
Let us now consider the family of area normalized 
cotan-Laplace operators $\Delta'_\varepsilon$, identified with
their corresponding matrices \eqref{matrix-cotan}. 
By the formula \eqref{matrix-cotan} we obtain
(we shall write $n$ for the number of edges emanating from each vertex
and abbreviate by $\dot{\Delta}'$ the derivative of $\Delta'_\varepsilon$ 
evaluated at $\varepsilon = 0$)
\begin{equation}\label{trzero2}
\begin{split}
\tr \dot{\Delta}' = \sum_{i \in V} (\dot{\Delta}')_{ii} &= 
 \frac{1}{A} \sum_{i \in V} \sum\limits_{(i,j) \in E} \dot{w}_{ij} 
-\sum_{i \in V} \frac{a_i}{A^2} \sum\limits_{(i,j) \in E} w_{ij}(0) \\
&=  \frac{1}{A} \sum_{i \in V} \sum\limits_{(i,j) \in E} \dot{w}_{ij} 
-\frac{2n}{A^2 \sqrt{3}}  \sum_{i \in V} a_i = 0,
\end{split}
\end{equation}
where the last equation follows from \eqref{trzero} and \eqref{asum}.
Similarly, we compute for the sum of the off-diagonal terms in the matrix \eqref{matrix-cotan}
\begin{equation}\label{offdiagonal2}
\begin{split}
\sum_{i \neq j} (\dot{\Delta}')_{ij} &= - 2 \sum\limits_{(i,j) \in E} \dot{w}_{ij} 
+ \frac{1}{A^2} \sum_{i \in V} a_i \sum\limits_{(i,j) \in E} w_{ij}(0) \\
&= - 2 \sum\limits_{(i,j) \in E} \dot{w}_{ij} 
+ \frac{2n}{A^2 \sqrt{3}} \sum_{i \in V} a_i 
= 0,
\end{split}
\end{equation}
where the last equality follows from from \eqref{offdiagonal} and \eqref{asum}.
The rest of the argument is exactly as in the previous case. 
\end{proof}

\section{Stationary point of determinants for strongly symmetric graphs}

We can now generalize Theorem \ref{sphth} to the case of strongly symmetric graphs.

\begin{theorem} \label{incomplete-th}
Let $M = (V, E, T)$ be a triangulation such that 
\begin{enumerate}
\item the graph $(V, E)$ is strongly symmetric as in Definition \ref{strongly-symmetric-def}.
\item $M$ has no boundary, i.e. each edge is bounded by exactly two triangles.
\item the discrete metric $l$ induces angles $\pi/3$ in each triangle.
\end{enumerate}
Consider a smooth family of discrete metrics $l(\varepsilon)$
such that $l(0) = l$. Consider the 
corresponding family of cotan-Laplace operators $\Delta_\varepsilon = \Delta_\text{cot} (l(\varepsilon))$ 
and their area normalized versions $\Delta'_\varepsilon$ 
as in Definition \ref{cotan-def}. 
\begin{enumerate}
\item[(i)] Then $\Delta_\varepsilon$ is stationary at $\varepsilon = 0$, i.e.
$$
\left. \frac{d}{d\varepsilon} \right|_{\varepsilon = 0} \log \Det \Delta_\varepsilon = 0.
$$
\item[(ii)] Assume the total area of triangles for $l(\varepsilon)$ is fixed. 
Then 
$$
\left. \frac{d}{d\varepsilon} \right|_{\varepsilon = 0} \log \Det \Delta'_\varepsilon = 0.
$$
\end{enumerate}
\end{theorem}

\begin{proof} 
The argument here differs from the proof in Theorem \ref{sphth}
only in the structure of $(\Delta_0+\delta)^{-1}$. For the properties 
of $\dot{\Delta}$ we still have \eqref{trzero}
and \eqref{offdiagonal}; for the properties 
of $\dot{\Delta}'$ we still have \eqref{trzero2}
and \eqref{offdiagonal2}, i.e.
$$
\tr \dot{\Delta} = \tr \dot{\Delta}' = 0, \quad 
\sum_{i \neq j} (\dot{\Delta})_{ij} = \sum_{i \neq j} (\dot{\Delta}')_{ij} = 0.
$$
By Lemma \ref{incomplete-lemma2}, 
for almost all $\delta > 0$ up to a discrete set of values,
the inverse $(\Delta_0 + \delta)^{-1}$
exists with all entries on the diagonal being equal, and all entries $(\Delta_0 + \delta)^{-1}_{ij}$
with $(i,j) \in E$ being equal. Same holds for $\Delta'_0$. Noting that 
$(\Delta_0)_{ij} = (\Delta'_0)_{ij} = 0$ for $(i,j) \notin E$, we conclude by Lemma \ref{matlemma} 
\begin{align*}
\tr \dot{\Delta} (\Delta_0 + \delta)^{-1} \equiv 0, \quad 
\tr \dot{\Delta}' (\Delta'_0 + \delta)^{-1} \equiv 0.
\end{align*}
Note as before in Theorem \ref{sphth}, that as long as all 
weights $w_{ij}> 0$ (which is the case for $\varepsilon$ sufficiently small
due to $w_{ij}(0) = \frac{2}{\sqrt{3}}$) the kernel of 
$\Delta_\varepsilon$ and of $\Delta'_\varepsilon$ 
equals the number of connected components of the graph, i.e. $1$ in this case. 
Thus, each $\Delta_\varepsilon, \Delta'_\varepsilon$ has the same rank and 
we obtain by Lemma \ref{logdet}
\begin{align*}
\left. \frac{d}{d\varepsilon} \right|_{0} \log \Det \Delta_\varepsilon  
= \lim_{\delta \to 0} \tr \dot{\Delta}_0 (\Delta_0 + \delta)^{-1} = 0, \\
\left. \frac{d}{d\varepsilon} \right|_{0} \log \Det \Delta'_\varepsilon  
= \lim_{\delta \to 0} \tr \dot{\Delta}'_0 (\Delta'_0 + \delta)^{-1} = 0.
\end{align*}
\end{proof}

\begin{remark}\label{remark-constant}
As is clear from the proof, Theorem \ref{incomplete-th} actually holds
for any triangulation $M = (V, E, T)$ such that, enumerating the vertices $V = \{1, \dots, n\}$,
the inverse matrix $(\Delta_0 + \delta)^{-1} = (c_{ij})$ 
satisfies 
\begin{align}\label{constant}
\forall i \in V: c_{ii} =x_0(\delta), \quad \forall (i,j) \in E: c_{ij} = x_1(\delta). 
\end{align}
for some functions $x_0(\delta)$ and $x_1(\delta)$, well-defined for
$\delta > 0$ at least up to some discrete number of exceptional values. 
The notion of strongly symmetric graphs in Definition \ref{strongly-symmetric-def}
is just one specific class of graphs that ensures \eqref{constant}, 
which may as well hold in a more general setting.
\end{remark}

\section{Numerical illustrations of the stationary points}

In order to understand the type of stationary points in 
Theorems \ref{sphth} and \ref{incomplete-th}, we perform some numerical calculations. 
We plot the discrete determinants under variation of only two edge lengths in the following cases:
\\[-2mm]

\begin{enumerate} 
\item the tetrahedron, $4$-point triangulation of the sphere 
$\mathbb{S}^2$ as in Figure \ref{pyradjazenz};
\item the $9$-point triangulation of the torus 
$\mathbb{T} = \mathbb{S}^2 \times \mathbb{S}^2$
as in Figure \ref{torus60bild}.
\end{enumerate} \ \\[-5mm] 

The plot is presented in Figure \ref{plot} using MATLAB to 
calculate the eigenvalues of the Laplacian in both cases numerically. 
The plots strongly suggest that the cotan-Laplacian attains a local minimum
when all edge lengths are equal.

\begin{figure}[h]
\centering
\includegraphics[scale=0.3]{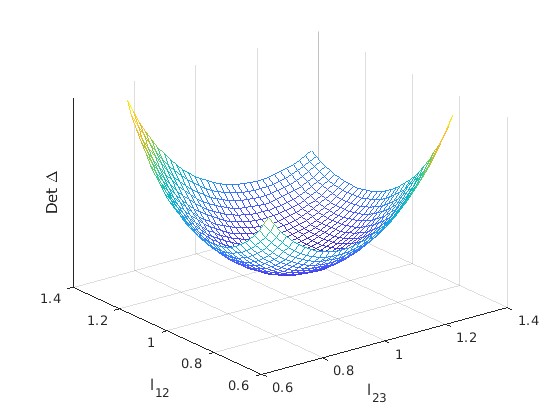}
\includegraphics[scale=0.3]{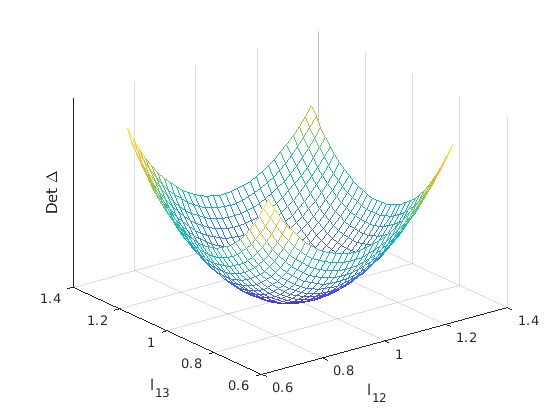}
\caption{Determinant of the cotan-Laplacian under variation 
of two side lengths. Left $-$ in the tetrahedron $4$-point triangulation of $\mathbb{S}^2$;
right $-$ in the $9$-point triangulation of $\mathbb{T} = \mathbb{S}^2 \times \mathbb{S}^2$.}
\label{plot}
\end{figure}

The variation of edge lengths near length $1$ preserves the triangle 
inequalities and thus defines a family of triangulations with discrete metrics. 
We want to point out that e.g. in the case of the tetrahedron, the variations 
can still be realized as convex polyhedrons. Indeed, it is known that for a 
triangulation endowed with a discrete metric 
$l$, the pair of conditions
\begin{enumerate}
\item The triangle inequalities hold,
\item The Cayley-Menger determinant is positive,
\end{enumerate}
are both necessary and sufficient for the existence of a convex polyhedron bearing edge lengths induced by $l$, see e.g. \cite{wirth}. 
The Cayley-Menger determinant for the tetrahedron is given by the determinant of the matrix
\[\begin{pmatrix}
0 & 1 & 1 & 1 & 1 \\
1 & 0 & l_{12} & l_{13} & l_{14} \\
1 & l_{21} & 0 & l_{23} & l_{24} \\
1 & l_{31} & l_{32} & 0 & l_{34} \\
1 & l_{41} & l_{42} & l_{43} & 0
\end{pmatrix}.\]
For $l_{ij} \equiv 1$ this determinant must be positive since we know that the uniform tetrahedron exists (in fact it equals $4$). Using the determinant's continuity in the matrix components, it immediately follows that the determinant is positive also in a neighbourhood of uniform edge lengths. Thus, arbitrary small variations of the uniform metric yield a realizable tetrahedron. 

\section{Discussion and outlook}

We close the discussion with a list of possible directions worth exploring. 

\subsection*{1) Different discrete determinants}

Our discussion centers around the cotan Laplacian $\Delta_\text{cot}$, 
motivated by the fact that it appears in the evolution equation for the 
curvature $\{K_i\}$ along discrete Ricci flow, as noted in \cite{confequiv}. 
There is a zoo of other meaningful discrete Laplacians to choose from, that 
may admit properties better aligned with the Osgood-Philips-Sarnak Theorem,
cf. \S \ref{minima-section}.

\subsection*{2) Graphs that are not strongly symmetric}

An obvious question is whether we can 
prove a discrete version of the Osgood-Philips-Sarnak Theorem
in a larger class of triangulations, in particular for those that are not necessarily
complete or strongly symmetric.

\subsection*{3) Discrete Ricci flow}

Ricci flow has been an essential ingredient in the proof of the 
Osgood-Philips-Sarnak Theorem. Can we employ the discrete Ricci flow 
to study discrete determinants as well? Is the discrete determinant monotone 
along discrete Ricci flow? Perhaps it would be necessary to 
shift the discussion to the use of circle packing metrics, 
for which some rigorous theory exists \cite{chowluo}.

\bibliographystyle{amsalpha}

\providecommand{\bysame}{\leavevmode\hbox to3em{\hrulefill}\thinspace}
\providecommand{\MR}{\relax\ifhmode\unskip\space\fi MR }
\providecommand{\MRhref}[2]{%
  \href{http://www.ams.org/mathscinet-getitem?mr=#1}{#2}
}
\providecommand{\href}[2]{#2}

\end{document}